\documentclass{amsart}

\makeatletter
\@namedef{subjclassname@2020}{%
  \textup{2020} Mathematics Subject Classification}
 
\makeatother 

\usepackage{amsthm,amssymb,amsfonts,latexsym,mathtools,thmtools,amsmath,lscape}
\usepackage[T1]{fontenc}
\usepackage{tikz-cd} 
\usepackage{enumitem} 
\usepackage{longtable}
\usepackage{multirow}
\usepackage{mathrsfs} 
\usepackage{multicol}
\usepackage{hyperref} 
\usepackage{hyperref}
\usepackage[all]{xy}
\usepackage{booktabs}
\hypersetup{
    colorlinks=true,
    linkcolor=blue,
    filecolor=blue,      
    urlcolor=blue,
    linktocpage=true
}

\newtheorem{theorem}{Theorem}[section]
\newtheorem{lemma}[theorem]{Lemma}
\newtheorem{proposition}[theorem]{Proposition}

\theoremstyle{definition}
\newtheorem{definition}[theorem]{Definition}

\newtheorem{remark}[theorem]{Remark}

\newtheorem{example}[theorem]{Example}

\theoremstyle{remark}

\numberwithin{equation}{section}

\begin{document}

\title[On differential smoothness of AS-regular algebras of dimension $5$]{On differential smoothness of certain \\ Artin-Schelter regular algebras of dimension $5$}


\author{Andr\'es Rubiano}
\address{Universidad ECCI}
\curraddr{Campus Universitario}
\email{arubianos@ecci.edu.co}
\thanks{}


\thanks{}



\date{}

\dedicatory{Dedicated to professor Armando Reyes}

\begin{abstract} 

This article investigates the differential smoothness of various five-dimensional Artin-Schelter regular algebras. By analyzing the relationship between the number of generators and the Gelfand-Kirillov dimension, we provide structural obstructions to differential smoothness in specific algebraic families. 
\end{abstract}

\maketitle


\section{Introduction}

Artin-Schelter regular algebras (AS-regular for short) were introduced by Artin and Schelter in their seminal work \cite{ArtinSchelter1987} as noncommutative analogues of polynomial rings, defined by three homological conditions: finite global dimension, finite Gelfand-Kirillov dimension, and the Gorenstein property. Since then, AS-regular algebras have become central objects in noncommutative algebraic geometry, serving as candidates for coordinate rings of noncommutative projective spaces.

Rogalski \cite{Rogalski2023} summarized and systematized the progress on the classification of connected graded AS-regular algebras of low global dimension. In particular, he provided a unifying exposition of the cases of dimensions $1$, $2$, and $3$, emphasizing their role as building blocks for the higher-dimensional theory. The classification can be summarized as follows: in dimension $1$, the only AS-regular algebra is the polynomial ring $\Bbbk[x]$; in dimension $2$, there are two families, namely the \emph{quantum plane} $A_q=\Bbbk\{x,y\}/(yx-qxy)$ for $q\in \Bbbk^\times$, and the \emph{Jordan plane} $A_J=\Bbbk\{x,y\}/(yx-xy-x^2)$; in dimension $3$, the classification is richer, since AS-regular algebras generated in degree one fall into families determined by quadratic relations linked to elliptic curves and twisted homogeneous coordinate rings, as first established in the foundational works of Artin, Tate, and Van den Bergh. Rogalski also emphasized the geometric interpretation of these algebras through point schemes, which connect their defining relations to algebraic curves. These low-dimensional cases serve as prototypes for the more intricate situation in dimension $4$ and for the still open and challenging classification problem in dimension $5$.

Recent research has produced significant progress in the classification of four-dimensional AS-regular algebras. One major breakthrough came from Zhang and Zhang \cite{ZhangZhang2008, ZhangZhang2009}, who introduced the notion of \emph{double Ore extensions} and proved that any connected graded double Ore extension of an AS-regular algebra is again AS-regular. Using this construction, they obtained $26$ distinct families of AS-regular algebras of global dimension four, all of which are strongly Noetherian, Auslander-regular, Koszul, and Cohen-Macaulay, and not isomorphic to either iterated Ore or normal extensions of three-dimensional AS-regular algebras. 

The five-dimensional case remains significantly more intricate and has motivated several research programs. Notable contributions include Fløystad and Vatne’s classification of $5$-dimensional AS-regular algebras with two generators \cite{FloystadVatne2009}, Zhou and Lu constructed a strongly noetherian and Cohen-Macaulay families using $\mathbb{Z}^2$-gradings \cite{ZhouLu2013}, and Wang and Wu classified two-generator AS-regular algebras of global dimension five using $A_{\infty}$-algebra methods \cite{WangWu2012}. Vancliff \cite{Vancliff2024} exhibited a quadratic five-dimensional AS-regular algebra that surprisingly admits no point modules. These diverse examples, distinguished by their generating sets, relation structures, and homological properties, highlight the richness and complexity of the five-dimensional setting.

The notion of differential smoothness has emerged as a framework for studying noncommutative analogues of smooth manifolds. Initiated by Brzezi\'nski and Sitarz \cite{BrzezinskiSitarz2017}, differential smoothness is defined via differential graded algebras of fixed dimension admitting a noncommutative version of the Hodge star isomorphism and a Poincar\'e-type duality between differential and integral forms. This notion extends classical smoothness into the noncommutative realm, and has been successfully applied to a wide range of algebras, including quantum spheres, quantum tori, Ore extensions, and skew polynomial algebras \cite{Brzezinski2015, DuboisViolette1988, ReyesSarmiento2022}.

Recent work has further developed sufficient and necessary conditions for differential smoothness in explicit families of noncommutative algebras, including biquadratic PBW algebras, diffusion algebras, and double Ore extensions of type $(14641)$ \cite{RubianoReyes2024DSBiquadraticAlgebras, RubianoReyes2024DSDoubleOreExtensions}. These results demonstrate how smoothness can be analyzed in terms of presentations and the construction of differential calculi compatible with algebraic relations.

Motivated by these advances, the purpose of this paper is to study differential smoothness in the setting of five-dimensional AS-regular algebras. In particular, we investigate how the variety of relation types and generating sets arising in dimension five influences the existence of differential calculi that certify smoothness. Our approach extends methods previously applied in lower dimensions and in skew PBW contexts.

The paper is organized as follows. Section \ref{PreliminariesDifferentialsmoothnessofbi-quadraticalgebras} contains definitions and preliminaries on differential smoothness of algebras; we also set up notation necessary for the rest of the paper. We review the key facts on some AS-regular algebras of dimension $5$ in order to set up notation and render this paper self-contained. In Section \ref{DifferentialandintegralcalculusAS5} we prove our main results, Theorem \ref{NoDS}, which states that no algebra with fewer generators than its Gelfand-Kirillov dimension can be differentially smooth and Theorem \ref{NoClifford}, which asserts the existence of an AS-regular algebra of global dimension $5$ that is differentiably smooth; in particular, the associated graded Clifford algebra $C$.

Throughout the paper, $\mathbb{N}$ denotes the set of natural numbers including zero. The word ring means an associative ring with identity not necessarily commutative. $Z(R)$ denotes the center of the ring $R$. All vector spaces and algebras (always associative and with unit) are over a fixed field $\Bbbk$. $\Bbbk^{*}$ denotes the non-zero elements of $\Bbbk$. As usual, the symbols $\mathbb{R}$ and $\mathbb{C}$ denote the fields of real and complex numbers, respectively. 

\section{Preliminaries}\label{PreliminariesDifferentialsmoothnessofbi-quadraticalgebras}

\subsection{Differential smoothness of algebras}\label{DefinitionsandpreliminariesDSA}




We follow Brzezi\'nski and Sitarz's presentation on differential smoothness carried out in \cite[Section 2]{BrzezinskiSitarz2017} (c.f. \cite{Brzezinski2008, Brzezinski2014}).

\begin{definition}[{\cite[Section 2.1]{BrzezinskiSitarz2017}}]\label{defDGA}
\begin{enumerate}
    \item [\rm (i)] A {\em differential graded algebra} is a non-negatively graded algebra $\Omega$ with the product denoted by $\wedge$ together with a degree-one linear map $d:\Omega^{\bullet} \to \Omega^{\bullet +1}$ that satisfies the graded Leibniz's rule 
    \[
    d(ab)=(da)b+(-1)^{\bullet}adb,\ a \in \Omega^{\bullet}, b\in \Omega
    \]
    
     \noindent and is such that $d \circ d = 0$.
    
    \item [\rm (ii)] A differential graded algebra $(\Omega(A), d)$ is a {\em calculus over an algebra} $A$ if $\Omega^0 (A) = A$ and $\Omega^n (A) = A\ dA \wedge dA \wedge \dotsb \wedge dA$ ($dA$ appears $n$-times) for all $n\in \mathbb{N}$ (this last is called the {\em density condition}). We write $(\Omega (A), d)$ with $\Omega (A) = \bigoplus_{n\in \mathbb{N}} \Omega^{n}(A)$. By using the Leibniz's rule, it follows that $\Omega^n (A) = dA \wedge dA \wedge \dotsb \wedge dA\ A$. A differential calculus $\Omega (A)$ is called {\em connected} if ${\rm ker}(d\mid_{\Omega^0 (A)}) = \Bbbk$.
    
    \item [\rm (iii)] A calculus $(\Omega (A), d)$ is said to have {\em dimension} $n$ if $\Omega^n (A)\neq 0$ and $\Omega^m (A) = 0$ for all $m > n$. An $n$-dimensional calculus $\Omega (A)$ {\em admits a volume form} if $\Omega^n (A)$ is isomorphic to $A$ as a left and right $A$-module. 
\end{enumerate}
\end{definition}

\begin{remark}
It is important to note that the product $\wedge$ defined on $\Omega(A)$ is not the usual exterior product; rather, it is merely a notation for this product.
\end{remark}

The existence of a right $A$-module isomorphism means that there is a free generator, say $\omega$, of $\Omega^n (A)$ (as a right $A$-module), i.e. $\omega \in \Omega^n (A)$, such that all elements of $\Omega^n (A)$ can be uniquely expressed as $\omega a$ with $a \in A$. If $\omega$ is also a free generator of $\Omega^n (A)$ as a left $A$-module, this is said to be a {\em volume form} on $\Omega (A)$.

The right $A$-module isomorphism $\Omega^n (A) \to A$ corresponding to a volume form $\omega$ is denoted by $\pi_{\omega}$, i.e.
\begin{equation}\label{BrzezinskiSitarz2017(2.1)}
\pi_{\omega} (\omega a) = a, \quad {\rm for\ all}\ a\in A.
\end{equation}

\noindent By using that $\Omega^n (A)$ is also isomorphic to $A$ as a left $A$-module, any free generator $\omega $ induces an algebra endomorphism $\nu_{\omega}$ of $A$ by the formula
\begin{equation}\label{BrzezinskiSitarz2017(2.2)}
    a \omega = \omega \nu_{\omega} (a).
\end{equation}

\noindent Note that if $\omega$ is a volume form, then $\nu_{\omega}$ is an algebra automorphism.

Now, we proceed to recall the key ingredients of the {\em integral calculus} on $A$ as dual to its differential calculus. For more details, see Brzezinski et al. \cite{Brzezinski2008, BrzezinskiElKaoutitLomp2010}.

Let $(\Omega (A), d)$ be a differential calculus on $A$. The space of $n$-forms $\Omega^n (A)$ is an $A$-bimodule. Consider $\mathcal{I}_{n}A$ the right dual of $\Omega^{n}(A)$, that is, $\mathcal{I}_{n}A := {\rm Hom}_{A}(\Omega^{n}(A),A)$. Notice that each of the $\mathcal{I}_{n}A$ is an $A$-bimodule with the actions
\begin{align*}
    (a\cdot\phi\cdot b)(\omega)=a\phi(b\omega),\quad {\rm for\ all}\ \phi \in \mathcal{I}_{n}A,\ \omega \in \Omega^{n}(A)\ {\rm and}\ a,b \in A.
\end{align*}

\noindent The direct sum of all the $\mathcal{I}_{n}A$, that is, $\mathcal{I}A = \bigoplus\limits_{n} \mathcal{I}_n A$, is a right $\Omega (A)$-module with action given by
\begin{align}\label{BrzezinskiSitarz2017(2.3)}
    (\phi\cdot\omega)(\omega')=\phi(\omega\wedge\omega'),\quad {\rm for\ all}\ \phi\in\mathcal{I}_{n + m}A, \ \omega\in \Omega^{n}(A) \ {\rm and} \ \omega' \in \Omega^{m}(A).
\end{align}

\begin{definition}[{\cite[Definition 2.1]{Brzezinski2008}}]
A {\em divergence} (also called {\em hom-connection}) on $A$ is a linear map $\nabla: \mathcal{I}_1 A \to A$ such that
\begin{equation}\label{BrzezinskiSitarz2017(2.4)}
    \nabla(\phi \cdot a) = \nabla(\phi) a + \phi(da), \quad {\rm for\ all}\ \phi \in \mathcal{I}_1 A \ {\rm and} \ a \in A.
\end{equation}  
\end{definition}

\noindent Note that a divergence can be extended to the whole of $\mathcal{I}A$, 
\[
\nabla_n: \mathcal{I}_{n+1} A \to \mathcal{I}_{n} A,
\]

\noindent by considering
\begin{equation}\label{BrzezinskiSitarz2017(2.5)}
\nabla_n(\phi)(\omega) = \nabla(\phi \cdot \omega) + (-1)^{n+1} \phi(d \omega), \quad {\rm for\ all}\ \phi \in \mathcal{I}_{n+1}(A)\ {\rm and} \ \omega \in \Omega^n (A).
\end{equation}

\noindent By putting together (\ref{BrzezinskiSitarz2017(2.4)}) and (\ref{BrzezinskiSitarz2017(2.5)}), we get the Leibniz's rule 
\begin{equation}
    \nabla_n(\phi \cdot \omega) = \nabla_{m + n}(\phi) \cdot \omega + (-1)^{m + n} \phi \cdot d\omega,
\end{equation}

\noindent for all elements $\phi \in \mathcal{I}_{m + n + 1} A$ and $\omega \in \Omega^m (A)$ \cite[Lemma 3.2]{Brzezinski2008}. In the case $n = 0$, if ${\rm Hom}_A(A, M)$ is canonically identified with $M$, then $\nabla_0$ reduces to the classical Leibniz's rule.



$\mathcal{I} A$ together with the $\nabla_n$ form a chain complex called the {\em complex of integral forms} over $A$. The cokernel map of $\nabla$, that is, $\Lambda: A \to {\rm Coker} \nabla = A / {\rm Im} \nabla$ is said to be the {\em integral on $A$ associated to} $\mathcal{I}A$.

Given a left $A$-module $X$ with action $a\cdot x$, for all $a\in A,\ x \in X$, and an algebra automorphism $\nu$ of $A$, the notation $^{\nu}X$ stands for $X$ with the $A$-module structure twisted by $\nu$, i.e. with the $A$-action $a\otimes x \mapsto \nu(a)\cdot x $.

The following definition of an \textit{integrable differential calculus} seeks to portray a version of Hodge star isomorphisms between the complex of differential forms of a differentiable manifold and a complex of dual modules of it \cite[p. 112]{Brzezinski2015}. 

\begin{definition}[{\cite[Definition 2.1]{BrzezinskiSitarz2017}}]\label{defInt}
An $n$-dimensional differential calculus $(\Omega (A), d)$ is said to be {\em integrable} if $(\Omega (A), d)$ admits a complex of integral forms $(\mathcal{I}A, \nabla)$ for which there exist an algebra automorphism $\nu$ of $A$ and $A$-bimodule isomorphisms \linebreak $\Theta_k: \Omega^{k} (A) \to ^{\nu} \mathcal{I}_{n-k}A$, $k = 0, \dotsc, n$, rendering commmutative the following diagram:
\[
\begin{tikzcd}
A \arrow{r}{d} \arrow{d}{\Theta_0} & \Omega^{1} (A) \arrow{d}{\Theta_1} \arrow{r}{d} & \Omega^2 (A)  \arrow{d}{\Theta_2} \arrow{r}{d} & \dotsb \arrow{r}{d} & \Omega^{n-1} (A) \arrow{d}{\Theta_{n-1}} \arrow{r}{d} & \Omega^n (A)  \arrow{d}{\Theta_n} \\ ^{\nu} \mathcal{I}_n A \arrow[swap]{r}{\nabla_{n-1}} & ^{\nu} \mathcal{I}_{n-1} A \arrow[swap]{r}{\nabla_{n-2}} & ^{\nu} \mathcal{I}_{n-2} A \arrow[swap]{r}{\nabla_{n-3}} & \dotsb \arrow[swap]{r}{\nabla_{1}} & ^{\nu} \mathcal{I}_{1} A \arrow[swap]{r}{\nabla} & ^{\nu} A
\end{tikzcd}
\]

\noindent The $n$-form $\omega:= \Theta_n^{-1}(1)\in \Omega^n (A)$ is called an {\em integrating volume form}. 
\end{definition}

The algebra of complex matrices $M_n(\mathbb{C})$ with the $n$-dimensional calculus generated by derivations presented by Dubois-Violette et al. \cite{DuboisViolette1988, DuboisVioletteKernerMadore1990}, the quantum group $SU_q(2)$ with the three-dimensional left covariant calculus developed by Woronowicz \cite{Woronowicz1987} and the quantum standard sphere with the restriction of the above calculus, are examples of algebras admitting integrable calculi. For more details on the subject, see Brzezi\'nski et al. \cite{BrzezinskiElKaoutitLomp2010}. 

The following proposition shows that the integrability of a differential calculus can be defined without explicit reference to integral forms. This allows us to guarantee the integrability by considering the existence of finitely generator elements that allow to determine left and right components of any homogeneous element of $\Omega(A)$.

\begin{proposition}\label{BrzezinskiSitarz2017Lemmas2.6and2.7}
\begin{enumerate}
\item [\rm (1)] \cite[Lemma 2.6]{BrzezinskiSitarz2017} Consider $(\Omega (A), d)$ an integrable and $n$-dimensional calculus over $A$ with integrating form $\omega$. Then $\Omega^{k} (A)$ is a finitely generated projective right $A$-module if there exist a finite number of forms $\omega_i \in \Omega^{k} (A)$ and $\overline{\omega}_i \in \Omega^{n-k} (A)$ such that, for all $\omega' \in \Omega^{k} (A)$, we have that 
\begin{equation*}
\omega' = \sum_{i} \omega_i \pi_{\omega} (\overline{\omega}_i \wedge \omega').
\end{equation*}

\item [\rm (2)] \cite[Lemma 2.7]{BrzezinskiSitarz2017} Let $(\Omega (A), d)$ be an $n$-dimensional calculus over $A$ admitting a volume form $\omega$. Assume that for all $k = 1, \ldots, n-1$, there exists a finite number of forms $\omega_{i}^{k},\overline{\omega}_{i}^{k} \in \Omega^{k}(A)$ such that for all $\omega'\in \Omega^k(A)$, we have that
\begin{equation*}
\omega'=\displaystyle\sum_i\omega_{i}^{k}\pi_\omega(\overline{\omega}_{i}^{n-k}\wedge\omega')=\displaystyle\sum_i\nu_{\omega}^{-1}(\pi_\omega(\omega'\wedge\omega_{i}^{n-k}))\overline{\omega}_{i}^{k},
\end{equation*}

\noindent where $\pi_{\omega}$ and $\nu_{\omega}$ are defined by {\rm (}\ref{BrzezinskiSitarz2017(2.1)}{\rm )} and {\rm (}\ref{BrzezinskiSitarz2017(2.2)}{\rm )}, respectively. Then $\omega$ is an integral form and all the $\Omega^{k}(A)$ are finitely generated and projective as left and right $A$-modules.
\end{enumerate}
\end{proposition}

Brzezi\'nski and Sitarz \cite[p. 421]{BrzezinskiSitarz2017} asserted that to connect the integrability of the differential graded algebra $(\Omega (A), d)$ with the algebra $A$, it is necessary to relate the dimension of the differential calculus $\Omega (A)$ with that of $A$, and since we are dealing with algebras that are deformations of coordinate algebras of affine varieties, the {\em Gelfand-Kirillov dimension} introduced by Gelfand and Kirillov \cite{GelfandKirillov1966, GelfandKirillov1966b} seems to be the best suited. Briefly, given an affine $\Bbbk$-algebra $A$, the {\em Gelfand-Kirillov dimension of} $A$, denoted by ${\rm GKdim}(A)$, is given by
\[
{\rm GKdim}(A) := \underset{n\to \infty}{\rm lim\ sup} \frac{{\rm log}({\rm dim}\ V^{n})}{{\rm log}\ n},
\]

\noindent where $V$ is a finite-dimensional subspace of $A$ that generates $A$ as an algebra. This definition is independent of choice of $V$. If $A$ is not affine, then its Gelfand-Kirillov dimension is defined to be the supremum of the Gelfand-Kirillov dimensions of all affine subalgebras of $A$. An affine domain of Gelfand-Kirillov dimension zero is precisely a division ring that is finite-dimensional over its center. In the case of an affine domain of Gelfand-Kirillov dimension one over $\Bbbk$, this is precisely a finite module over its center, and thus polynomial identity. In some sense, this dimensions measures the deviation of the algebra $A$ from finite dimensionality. For more details about this dimension, see the treatment developed by Krause and Lenagan \cite{KrauseLenagan2000}.

After preliminaries above, we arrive to the key notion of this paper.

\begin{definition}[{\cite[Definition 2.4]{BrzezinskiSitarz2017}}]\label{BrzezinskiSitarz2017Definition2.4}
An affine algebra $A$ with integer Gelfand-Kirillov dimension $n$ is said to be {\em differentially smooth} if it admits an $n$-dimensional connected integrable differential calculus $(\Omega (A), d)$.
\end{definition}

From Definition \ref{BrzezinskiSitarz2017Definition2.4} it follows that a differentially smooth algebra comes equipped with a well-behaved differential structure and with the precise concept of integration \cite[p. 2414]{BrzezinskiLomp2018}.

\begin{example}
As we said in the Introduction, several examples of noncommutative algebras have been proved to be differentially smooth (e.g. \cite{Brzezinski2015, BrzezinskiElKaoutitLomp2010, BrzezinskiLomp2018, BrzezinskiSitarz2017, Karacuha2015, KaracuhaLomp2014, ReyesSarmiento2022}). For instance, the polynomial algebra $\Bbbk[x_1, \dotsc, x_n]$ has the Gelfand-Kirillov dimension $n$ and the usual exterior algebra is an $n$-dimensional integrable calculus, whence $\Bbbk[x_1, \dotsc, x_n]$ is differentially smooth. From \cite{BrzezinskiElKaoutitLomp2010} we know that the coordinate algebras of the quantum group $SU_q(2)$, the standard quantum Podle\'s and the quantum Manin plane are differentially smooth.
\end{example}

\begin{remark}
As expected, there are examples of algebras that are not differentially smooth. Consider the commutative algebra $A = \mathbb{C}[x, y] / \langle xy \rangle$. A proof by contradiction shows that for this algebra there are no one-dimensional connected integrable calculi over $A$, so this provides a structural obstruction to differential smoothness \cite[Example 2.5]{BrzezinskiSitarz2017}.
\end{remark}

\subsection{Artin-Schelter regular algebras of dimension \texorpdfstring{$5$}{Lg}}\label{AS5}

In the setting of noncommutative algebras appearing in noncommutative geometry, {\em Artin-Schelter regular algebras} introduced by Artin and Schelter \cite{ArtinSchelter1987} as a class of three-dimensional graded algebras may be regarded as noncommutative versions of the polynomial ring in three indeterminates.

\begin{definition}\label{ASSchelterregularalgebradefinition}
An {\em Artin-Schelter} (AS-regular for short) {\em regular algebra} is an $\mathbb{N}$-graded algebra $A = \bigoplus\limits_{n\ge 0} A_n$ over $\Bbbk$ which is {\em connected} (that is, $A_0 = \Bbbk$) and satisfies the following three conditions:
\begin{enumerate}
    \item [\rm (i)] $A$ has finite global dimension $d$, that is, ${\rm gld}(A) = d < \infty$.
    \item [\rm (ii)] $A$ has polynomial growth, i.e., ${\rm GKdim}(A) < \infty$.
    \item [\rm (iii)] ({\em Gorenstein condition}) $A$ is {\em Gorenstein} in the sense that 
    \[
    {\rm Ext}_A^{i} (_A\Bbbk,\ _AA) \cong \begin{cases}
0, & {\rm if}\ i\neq d, \\ \Bbbk(l)_A, & {\rm if}\ i = d, 
    \end{cases}
    \]

    \noindent for some shift $l\in \mathbb{Z}$. Equivalently, 
    \[
    {\rm Ext}_A^{i} (\Bbbk_A , A_A) \cong \begin{cases}
0, & {\rm if}\ i\neq d, \\ _A \Bbbk(l), & {\rm if}\ i = d, 
    \end{cases}
    \]

     \noindent for some shift $l\in \mathbb{Z}$ \cite{Rogalski2023}. 
\end{enumerate}
\end{definition}

These conditions capture both homological regularity and finiteness in growth, making AS-regular algebras behave similarly to commutative polynomial rings in many respects. Perhaps most significantly, these algebras often serve as coordinate rings of noncommutative analogues of projective spaces.

\begin{example}
\leavevmode
\begin{enumerate}
  \item The commutative polynomial ring \( \Bbbk[x_1, x_2, \dots, x_n] \), graded by total degree, is an AS-regular algebra of global dimension \( n \). It satisfies the Gorenstein condition via classical local duality and has GK dimension \( n \).

  \item The three-dimensional Sklyanin algebra,
  {\small{\[
    A = \Bbbk\{ x, y, z \} / \big\langle ayz + bzy + cx^2,\ azx + bxz + cy^2,\ axy + byx + cz^2 \big\rangle,
  \]}}
\noindent defined for suitable parameters \( a, b, c \in \Bbbk \), is a noncommutative AS-regular algebra of global dimension $3$. This algebra was introduced by Artin, Tate, and Van den Bergh as a deformation of the commutative polynomial ring and has become a cornerstone example in the study of noncommutative projective geometry \cite{ArtinTateVandenBergh1990}.
\end{enumerate}
\end{example}

We focus on the study of five-dimensional AS-regular algebras, especially those generated by two, three and five elements.

\subsubsection{AS-regular algebras with two generators}

In this case, the classification of five-dimensional AS-regular algebras generated by two elements under a $\mathbb{Z}^2$-grading is undertaken in \cite{ZhouLu2013} using the method of Hilbert-driven Gröbner basis computations.

\begin{proposition}[{\cite[Theorem A]{ZhouLu2013}}]\label{ZhouLu2013ThmA}
Let $\mathcal{X}$ denote the family of algebras labeled from $A$ to $P$ constructed in Sections $3$-$7$ \cite{ZhouLu2013}. Then $\mathcal{X}$ is, up to isomorphism and switching, a complete list of properly $\mathbb{Z}^2$-graded AS-regular algebras of global dimension five which are domains generated by two elements and have Gelfand-Kirillov dimension at least four.
\end{proposition}

\begin{proposition}[{\cite[Theorem B]{ZhouLu2013}}]\label{ZhouLu2013ThmB}
Let $A$ be a properly $\mathbb{Z}^2$-graded AS-regular algebra of global dimension five generated by two elements. If $A$ is a domain and $\mathrm{GKdim}(A) \geq 4$, then $A$ is strongly noetherian, Auslander regular, and Cohen-Macaulay.
\end{proposition}

\subsubsection{AS-regular algebras with four generators}

A distinct and constructive approach to generating AS-regular algebras of global dimension five was developed by Li and Wang in \cite{LiWang2016}, where they consider deformations of the universal enveloping algebra of a finite-dimensional graded Lie algebra. 

Let \( A = \Bbbk \{ x, y, z, w \} / \langle r_1, r_2, r_3, r_4, r_5\rangle \), where:
\[
\begin{aligned}
r_1 &= xy + \alpha yx + \beta zw + \gamma wz, \\
r_2 &= xz + a zx, \quad
r_3 = xw + b wx, \\
r_4 &= yz + c zy, \quad
r_5 = yw + d wy,
\end{aligned}
\]
\noindent with parameters \( \alpha, \beta, \gamma, a, b, c, d \in \Bbbk \). The authors prove that when all parameters are nonzero and satisfy \( abcd = 1 \), the algebra \( A \) is AS-regular of global dimension $5$.

\begin{proposition} [{\cite[Theorem 2.9]{LiWang2016}}]\label{LiWang2016Thm2.9}
Let \( A \) be as defined above. If \( \alpha\beta\gamma \neq 0 \) and \( abcd = 1 \), then \( A \) is an AS-regular algebra of global dimension $5$, which is also strongly noetherian, Auslander regular, and Cohen-Macaulay. Moreover, its Hilbert series is
\[
H_A(t) = \frac{1}{(1 - t)^4(1 - t^2)}.
\]
\end{proposition}


\subsubsection{AS-regular algebras with five generators}

The construction relies on a family of symmetric \(5 \times 5\) matrices \(M_1, \ldots, M_5\) over an algebraically closed field $\Bbbk$, with \(\mathrm{char}(\Bbbk) \neq 2\), defining a system of quadrics:
\[
\begin{aligned}
q_1 &= z_1^2 + z_2z_4, \quad
q_2 = z_2^2 + z_1z_3, \quad
q_3 = z_3^2 + z_1z_4, \\
q_4 &= z_4^2 + z_1z_5, \quad
q_5 = z_5^2 + z_1z_2.
\end{aligned}
\]
The associated graded Clifford algebra \(C\) is then defined by generators \(x_1, \ldots, x_5\) of degree one and quadratic relations derived from the \(q_i\). The defining relations of \(C\) include:
\begin{align}
x_1x_2 + x_2x_1 &= x_5^2, & x_2x_5 + x_5x_2 &= 0, \notag \\
x_1x_3 + x_3x_1 &= x_2^2, & x_3x_4 + x_4x_3 &= 0, \notag \\
x_1x_4 + x_4x_1 &= x_3^2, & x_3x_5 + x_5x_3 &= 0, \label{relC}\\
x_1x_5 + x_5x_1 &= x_4^2, & x_4x_5 + x_5x_4 &= 0,\notag \\
x_2x_4 + x_4x_2 &= x_1^2, & x_2x_3 + x_3x_2 &= 0. \notag
\end{align}
\begin{proposition}[{\cite[p.4]{Vancliff2024}}]\label{AS5noPoints}
Let \(C\) be the graded Clifford algebra constructed as above. Then \(C\) is a quadratic AS-regular algebra of global dimension five, which is Auslander-regular, Cohen-Macaulay, and has Hilbert series \(H_C(t) = \frac{1}{(1 - t)^{5}}\). Moreover, \(C\) admits no point modules.
\end{proposition}

\begin{remark}
This construction answers an open question raised during the preparation of D. Rogalski’s survey on AS-regular algebras \cite{Rogalski2023}, where the existence of such an example was anticipated but not explicitly known in the literature. The methods of Vancliff build on earlier work with Van den Bergh and Le Bruyn, and demonstrate that quadratic regularity does not guarantee the presence of point modules when the dimension is sufficiently high.
\end{remark}

\section{Differential and integral calculus}\label{DifferentialandintegralcalculusAS5}

This section contains the important results of the paper.

\begin{definition}\label{genwis}
Let $A$ be an algebra with $n$ generators $x_1, \ldots, x_n$. Consider a family of automorphisms of $A$, $\{\nu_{x_1},\ldots, \nu_{x_n}\}$ of the form
$$
\nu_{x_j}(x_i)=a_{ij}x_i+b_{ij} \quad \text{for all } 1 \leq i, j \leq n,,
$$
where $a_{ij} \in \Bbbk$ and $b_{ij} \in \Bbbk$.

We say that the family of automorphisms $\{\nu_{x_1},\ldots, \nu_{x_n}\}$ is {\em generator-wise non-reflective} if $a_{ii}=1$, for all $1\leq i \leq n$.
 
If a family of automorphisms does not comply with this, it is said to be {\em generator-wise reflective}.
\end{definition}

\begin{lemma}\label{dx0}
Let $A$ be an algebra with $n$ generators $x_1, \ldots, x_n$, and let $\{\nu_{x_1},\ldots,\nu_{x_n}\}$ be a generator-wise reflective family of automorphisms of $A$ with the form $\nu_{x_i}(x_j)=a_{ij}x_j+b_{ij}$, $1 \leq i,j \leq n$. Consider ($\Omega^1(A), d)$ the first-order differential calculus on $A$. Then there exists $k \in \{1,\ldots,n\}$ such that
\[
dx_k \wedge dx_k = 0.
\]
\end{lemma}

\begin{proof}
Since the family of automorphisms is generator-wise reflective, there exists 
$k \in \{1,\ldots,n\}$ such that
\[
\nu_{x_k}(x_k) \;=\; a_{kk} x_k + b_{kk},
\qquad a_{kk} \notin \{0,1\},\ b_{kk} \in \Bbbk.
\]
Consider now the element
\[
x_k dx_k \;=\; (dx_k)\,\nu_{x_k}(x_k)
\;=\; dx_k(a_{kk}x_k + b_{kk}).
\]
Applying $d$ to both sides of this identity, we obtain
\begin{align*}
d(x_k dx_k) 
  &= d\big(dx_k(a_{kk}x_k + b_{kk})\big) \\
dx_k \wedge dx_k + x_k d^2(x_k) 
  &= a_{kk} d^2(x_k) x_k + a_{kk}\, dx_k \wedge dx_k + d(b_{kk}) \\
dx_k \wedge dx_k 
  &= a_{kk}\, dx_k \wedge dx_k \\
(1 - a_{kk})\, dx_k \wedge dx_k 
  &= 0.
\end{align*}
Since $1 - a_{kk}$ is invertible in $\Bbbk$, it follows that $dx_k \wedge dx_k = 0$.

\end{proof}

\begin{theorem}\label{NoDS}
Let $A$ be a $\Bbbk$-algebra generated by elements $x_1,\ldots,x_n$ and assume that every family of automorphisms $\{\nu_{x_1},\ldots,\nu_{x_n}\}$ associated with a first-order differential calculus on $A$ is generator-wise reflective. If $\mathrm{GKdim}(A)=m>n$, then $A$ is not differentially smooth.
\end{theorem}

\begin{proof}
In order to examine the differentiable smoothness of \( A \), it is necessary that a differential calculus of dimension \( m \), denoted \( \Omega^m(A) \), can be constructed.

According to Definition~\ref{defDGA}~(ii), the module \( \Omega^m(A) \) is given by
\[
\Omega^m(A) = A dA \wedge dA \wedge \cdots \wedge dA \quad \text{($m$-times)},
\]
that is, it is the \( m \)-fold product of the module of first-order differential forms. Hence, once the space \( dA \) is known, \( \Omega^m(A) \) is, in principle, determined.

However, since \( x_1, \ldots, x_n \) generate \( A \), it follows that the differentials \( dx_1, \ldots, dx_n \) generate \( dA \) as an \( A \)-bimodule. Consequently, $dA \wedge \cdots \wedge dA \quad \text{($m$ factors)}$ is generated by products of the elements \( dx_1, \ldots, dx_n \). But given that \( m > n \), it is inevitable that at least \( m - n \) of these \( dx_i \) will be repeated in any such product.

The relations among the generators $dx_1,\dots,dx_n$ of $\Omega^1(A)$ are of the following form:
\[
dx_i\wedge dx_j \;=\; a_{ij}\,dx_j\wedge dx_i,
\qquad \text{for all } 1\le i,j\le n.
\]

Since the family of automorphisms is generator-wise reflective, Lemma \ref{dx0} ensures the existence of a differential $dx_k$ such that 
\[
dx_k \wedge dx_k = 0
\]
for some $k \in \{1,\ldots,n\}$. Therefore, any product of \( m \) factors taken from  the set \( \{dx_1, \ldots, dx_n\} \), with \( m > n \), must vanish due to the repetition of at least one \( dx_k \). It follows that
\[
\Omega^m(A) = 0 \quad \text{whenever } m > n.
\]

Therefore, if \(\mathrm{GKdim}(A) = m > n\), then \(\Omega^m(A) = 0\), and thus no connected differential calculus of dimension \(m\) exists over \(A\). Consequently, \(A\) is not differentially smooth.

\end{proof}

\begin{remark}\label{rem:partial-result}
The obstruction provided by Theorem~\ref{NoDS} should be regarded as a partial result. Indeed, the statement is proved under the additional restriction that the differential calculus is compatible with a prescribed (and rather specific) family of automorphisms of $A$.

In general, the structure of $\mathrm{Aut}(A)$ for Artin--Schelter regular algebras of higher dimension can be highly nontrivial, and the systematic search and classification of automorphisms remains an active and, in many cases, open direction of research. Consequently, without a priori information on the automorphism group (or on which automorphisms may arise from an integrable calculus), it is currently not possible to upgrade the above obstruction to a fully general statement.
\end{remark}

We now consider families of five-dimensional AS-regular algebras where \( m = 5 \) and \( n < 5 \). For example, in \cite{ZhouLu2013}, Zhou and Lu classify all properly \( \mathbb{Z}^2 \)-graded AS-regular domains of global dimension 5 that are generated by two elements. Each algebra \( A \in \mathcal{X} \) in their classification satisfies:
\[
\mathrm{gldim}(A) = \mathrm{GKdim}(A) = 5, \quad \text{and} \quad \mathrm{Gen}(A) = 2.
\]
Similarly, in \cite{WangWu2012}, Wang and Wu construct nine families of AS-regular algebras of dimension five, each generated by two degree-one generators and defined by three relations of degree four. Again, each such algebra \( B \) satisfies:
\[
\mathrm{gldim}(B) = \mathrm{GKdim}(B) = 5, \quad \mathrm{Gen}(B) = 2.
\]

In both cases, Theorem~\ref{NoDS} applies once we ensure that every family of automorphisms 
$\{\nu_{x}, \nu_{y}\}$ constructed for the first-order differential calculus is generator-wise reflective. 
Since $m = 5$ and $n = 2 < 5$, any attempt to construct a $5$-dimensional differential calculus on these 
algebras necessarily yields $\Omega^{5}(A) = 0$. Consequently, this provides a structural obstruction to 
their differential smoothness.

Another example arises from the work of Li and Wang \cite{LiWang2016}, where an AS-regular algebra \( C \) is constructed from deformations of enveloping algebras of graded Lie algebras. The algebra \( C \) is defined by five quadratic relations among four generators:
\[
\mathrm{Gen}(C) = 4, \quad \mathrm{GKdim}(C) = 5.
\]
By Theorem~\ref{NoDS}, once we ensure that every family of automorphisms 
$\{\nu_{x}, \nu_{y}, \nu_{z}, \nu_{w}\}$ constructed for the first-order differential calculus 
is generator-wise reflective, we again conclude that $\Omega^{5}(C) = 0$.

Although these AS-regular algebras satisfy strong homological conditions (such as being noetherian, Auslander regular, and Cohen-Macaulay) they inherently fail to admit a nontrivial differential calculus of dimension five due to a mismatch between the number of generators and their Gelfand-Kirillov dimension. This phenomenon highlights a fundamental obstruction to differential smoothness in high-dimensional settings.

The quadratic Clifford algebra constructed by Vancliff in \cite{Vancliff2024} is an AS‑regular algebra of global dimension $5$ with
\[
H_C(t) = \frac{1}{(1 - t)^5},
\]
implying that \(\mathrm{GKdim}(C) = 5\). Hence the number of generators is exactly equal to the GK‑dimension and Theorem \ref{NoDS} does not apply here. However, this algebra is differentially smooth, as will be shown in the following result.

\begin{theorem}\label{NoClifford}
The associated graded Clifford algebra \(C\) with relations \ref{relC} is differentially smooth.
\end{theorem}
\begin{proof}
Consider the following automorphisms:
\begin{align}
   \nu_{x_i}(x_j) = &\ -x_j, \text{ for all } 1\leq i,j \leq 5. \label{Autos}
\end{align}
The maps $\nu_{x_i}$, $1 \leq i \leq 5$ can be extended to an algebra homomorphism of $C$ if and only if the definitions of homomorphism over the generators $x_j$, $1 \leq j \leq 5$ respect relations {\rm (}\ref{relC}{\rm )}. It is readily seen that the condition is satisfied unconditionally. Furthermore, these automorphisms mutually commute, as each acts on every indeterminate by multiplication with a scalar.

Consider $\Omega^{1}(C)$ a free right $C$-module of rank five with generators $dx_j$, $1\leq j \leq 5$. For all $p\in C$ define a left $C$-module structure by
\begin{align}
    pdx_j = &\ (dx_j)\nu_{x_j}(p), \ 1 \leq j \leq 5 \label{relrightmod}.
\end{align}
The relations in $\Omega^{1}(C)$ are given by 
\begin{align*}
    x_jdx_i = -(dx_i)x_j, \text{ for all } 1\leq i,j\leq 5. 
\end{align*}
We want to extend the correspondences 
\begin{equation*}
x_j \mapsto dx_j, \ 1\leq j\leq 5, 
\end{equation*} 
to a map $d: C \to \Omega^{1}(C)$ satisfying the Leibniz's rule. This is possible if it is compatible with the nontrivial relations {\rm (}\ref{relC}{\rm )}, i.e., if the equalities
\begin{align*}
(dx_1)x_2 + x_1dx_2 + (dx_2)x_1 + x_2dx_1 &= (dx_5)x_5+x_5dx_5, \\ 
(dx_1)x_3 + x_1dx_3 + (dx_3)x_1 + x_3dx_1 &= (dx_2)x_2+x_2dx_2, \\ 
(dx_1)x_4 + x_1dx_4 + (dx_4)x_1+ x_4dx_1 &= (dx_3)x_3+x_3dx_3, \\ 
(dx_1)x_5 + x_1dx_5 +(dx_5)x_1+x_5dx_1 &= (dx_4)x_4+x_4dx_4, \\
(dx_2)x_4 +x_2dx_4+ (dx_4)x_2+ x_4dx_2 &= (dx_1)x_1+x_1dx_1, \\ 
(dx_2)x_5 + x_2dx_5 +(dx_5)x_2+x_5dx_2 &= 0, \\
(dx_3)x_4 +x_3dx_4+ (dx_4)x_3+ x_4dx_3 &= 0,  \\
(dx_3)x_5 + x_3dx_5 +(dx_5)x_3+x_5dx_3 &= 0, \\
(dx_4)x_5 + x_4dx_5 +(dx_5)x_4+x_5dx_4 &= 0,\\
(dx_2)x_3 + x_2dx_3 + (dx_3)x_2 + x_3dx_2 &= 0. 
\end{align*}
hold. 

Define $\Bbbk$-linear maps 
\begin{equation*}
\partial_{x_i}: C \rightarrow C, \ 1 \leq i \leq 5,
\end{equation*}

such that

\begin{align*}
   d(a) & =(dx_1)\partial_{x_1}(a)+(dx_2)\partial_{x_2}(a)+(dx_3)\partial_{x_3}(a) \\
   & +(dx_4)\partial_{x_4}(a)+(dx_5)\partial_{x_5}(a), \ {\rm for\ all} \ a \in C.
\end{align*}

Since $dx_i$, $1 \leq i \leq 5$, are free generators of the right $C$-module $\Omega^1(C)$, these maps are well-defined. Note that $d(a)=0$ if and only if $\partial_{x_i}(a)=0$, for all $1\leq i \leq 5$. By using the four relations in {\rm (}\ref{relrightmod}{\rm )} and the definitions of the maps $\nu_{x_i}$, $1\leq i \leq 5$, we get that
\begin{align*}
\partial_{x_1}(x_1^{k_1}x_2^{k_2}x_3^{k_3}x_4^{k_4}x_5^{k_5}) = &\ (-1)^{k_1}x_1^{k_1-1}x_2^{k_2}x_3^{k_3}x_4^{k_4}x_5^{k_5},  \\
\partial_{x_2}(x_1^{k_1}x_2^{k_2}x_3^{k_3}x_4^{k_4}x_5^{k_5}) = &\  (-1)^{k_1+k_2}x_1^{k_1}x_2^{k_2-1}x_3^{k_3}x_4^{k_4}x_5^{k_5},  \\
\partial_{x_3}(x_1^{k_1}x_2^{k_2}x_3^{k_3}x_4^{k_4}x_5^{k_5}) = &\ (-1)^{k_1+k_2+k_3}x_1^{k_1}x_2^{k_2}x_3^{k_3-1}x_4^{k_4}x_5^{k_5},  \\
\partial_{x_4}(x_1^{k_1}x_2^{k_2}x_3^{k_3}x_4^{k_4}x_5^{k_5}) = &\ (-1)^{k_1+k_2+k_3+k_4}x_1^{k_1}x_2^{k_2}x_3^{k_3}x_4^{k_4-1}x_5^{k_5},  \\
\partial_{x_5}(x_1^{k_1}x_2^{k_2}x_3^{k_3}x_4^{k_4}x_5^{k_5}) = &\ (-1)^{k_1+k_2+k_3+k_4+k_5}x_1^{k_1}x_2^{k_2}x_3^{k_3}x_4^{k_4}x_5^{k_5-1}  .
\end{align*}
In this way, $d(a)=0$ if and only if $a$ is a scalar multiple of the identity.

Now, we have that higher-order calculi are defined as follows:
\[
\Omega^2(C) = \left[\bigoplus_{\substack{i,j=1,\\ i <j}}^{5}dx_i\wedge dx_{j} \right]C,
\]
\[
\Omega^3(C) = \left[\bigoplus_{\substack{i,j,k=1,\\ i <j<k}}^{5}dx_i\wedge dx_{j}\wedge dx_k \right]C,
\]
\[
\Omega^4(C) = \left[\bigoplus_{r=1}^{4}dx_1\wedge \cdots dx_{r-1}\wedge dx_{r+1} \wedge \cdots \wedge dx_{5} \right]C.
\]
\[
\Omega^5(C) = dx_1\wedge dx_2 \wedge dx_3 \wedge dx_4 \wedge dx_5
\]

The diagram of Definition \ref{defInt} gives an universal extension of $d$ to higher forms compatible with {\rm (}\ref{relC}{\rm )}. This leads to the following rules for $\Omega^l(C)$  $(2\leq l\leq 4)$:
\begin{align}
\bigwedge_{k=1}^{l}dx_{q(k)} = &\ (-1)^{\sharp}\bigwedge_{k=1}^ldx_{p(k)}
\end{align}

\noindent where 
$$
q:\{1,\ldots,l\}\rightarrow \{1,\ldots,5\}
$$ 

\noindent is an injective map and 
$$
p:\{1,\ldots,l\}\rightarrow \text{Im}(q)$$

\noindent is an increasing injective map and $\sharp$ is the number of $2$-permutations needed to transform $q$ into $p$.

Since $\Omega^5(C) = \omega C\cong C$ as a right and left $C$-module, with $\omega=dx_1\wedge \cdots \wedge dx_5$, where $\nu_{\omega}=\nu_{x_1}\circ\cdots\circ\nu_{x_5}$, we have that $\omega$ is a volume form of $C$. From Proposition \ref{BrzezinskiSitarz2017Lemmas2.6and2.7} (2) we get that $\omega$ is an integral form by setting
\begin{align*}
    \omega_i^j = &\ \bigwedge_{k=1}^{j}dx_{p_{i,j}(k)}, \text{ for } 1\leq i \leq \binom{5}{j-1}, \\
    \bar{\omega}_i^{5-j} = &\ (-1)^{\sharp_{i,j}}\bigwedge_{k=j+1}^{n}dx_{\bar{p}_{i,j}(k)}, \text{ for } 1\leq i \leq \binom{5}{j-1},
\end{align*}
for $1\leq j \leq 5$ and where 
\begin{align*}
    p_{i,j}:\{1,\ldots,j\}\rightarrow &\ \{1,\ldots,5\}, \quad {\rm and} \\
\bar{p}_{i,j}:\{j+1,\ldots,5\}\rightarrow &\ (\text{Im}(p_{i,j}))^c
\end{align*}

\noindent (the symbol $\square^c$ denotes the complement of the set $\square$), are increasing injective maps, and $\sharp_{i,j}$ is the number of $2$-permutation needed to transform 
\[
\left\{\bar{p}_{i,j}(j+1),\ldots, \bar{p}_{i,j}(5), p_{i,j}(1), \ldots, p_{i,j}(j)\right\} \quad {\rm into\ the\ set} \quad \{1, \ldots, 5\}.
\]
Consider $\omega' \in \Omega^j(C)$, that is,  
\begin{align*}
\omega' =\sum_{i=1}^{\binom{5}{j-1}}\bigwedge_{k=1}^{j}dx_{p_{i,j}(k)}b_i, \quad {\rm with} \  n_i \in \Bbbk.
\end{align*}

\noindent Then
\begin{align*}
 \sum_{i=1}^{\binom{5}{j-1}}\omega_{i}^{j}\pi_{\omega}(\bar{\omega}_i^{5-j}\wedge \omega') = &\ \sum_{i=1}^{\binom{5}{j-1}}\left[\bigwedge_{k=1}^{j}dx_{p_i(k)}\right] \cdot  \pi_{\omega} \left[(-1)^{\sharp_{i,j}} \bigwedge_{k=j+1}^{n}dx_{\bar{p}_{i,j}(k)} \wedge \omega'\right] \\
 = &\ \displaystyle \sum_{i=1}^{\binom{5}{j-1}}\bigwedge_{k=1}^{j}dx_{p_{i,j}(k)}n_i =  \omega'.
\end{align*}

By Proposition \ref{BrzezinskiSitarz2017Lemmas2.6and2.7} (2), it follows that $A$ is differentially smooth.
\end{proof}

\begin{remark}
The classification of five-dimensional AS-regular algebras remains an active area of research, and the number of known explicit examples (especially those with five generators) is still quite limited in the literature. Consequently, the fact that the Clifford algebra \( C \) on Theorem \ref{NoClifford} is differentially smooth cannot be interpreted as a general property shared by all AS-regular algebras of global dimension five with five generators. Such smoothness depends intricately on the specific form of the defining relations and on the structure of the algebra’s automorphism group.
\end{remark}


\begin{thebibliography}{48}

\bibitem{ArtinSchelter1987} M. Artin and W. F. Schelter, Graded algebras of global dimension 3, {\em Adv. Math.} {\bf 66}(2) (1987) 171--216.


\bibitem{ArtinTateVandenBergh1990} M. Artin, J. Tate, and M. Van den Bergh, Some algebras associated to automorphisms of elliptic curves. in \textit{The Grothendieck Festschrift}, Vol. I, Progr. Math. \textbf{86}, Birkhäuser Boston, Boston, MA, (1990) 33--85.


\bibitem{Brzezinski2008} T. Brzeziński, Noncommutative Connections of The Second Kind, {\em J. Algebra Appl.} {\bf 7}(5) (2008) 557--573.


\bibitem{Brzezinski2014} T. Brzeziński, On the Smoothness of the Noncommutative Pillow and Quantum Teardrops, {\em SIGMA} {\bf 10}(015) (2014) 1--8.

\bibitem{Brzezinski2015} T. Brzeziński, Differential smoothness of affine Hopf algebras of Gelfand-Kirillov of dimension two, {\em Colloq. Math.} {\bf 139}(1) (2015) 111--119.


\bibitem{BrzezinskiElKaoutitLomp2010} T. Brzeziński, L. El Kaoutit, and C. Lomp, Noncommutative integral forms and twisted multi-derivations, {\em J. Noncommut. Geom.} {\bf 4}(2) (2010) 281--312.

\bibitem{BrzezinskiLomp2018} T. Brzeziński and C. Lomp, Differential smoothness of skew polynomial rings, {\em J. Pure Appl. Algebra} {\bf 222}(9) (2018) 2413--2426.

\bibitem{BrzezinskiSitarz2017}T. Brzezi{\'n}ski and A. Sitarz, Smooth geometry of the noncommutative pillow, cones and lens spaces, {\em J. Noncommut. Geom.} {\bf 11}(2) (2017) 413--449.



\bibitem{DuboisViolette1988} M. Dubois-Violette, Dérivations et calcul différentiel non commutatif, {\em C. R. Acad. Sci. Paris, Ser. I} {\bf 307} (1988) 403--408.


\bibitem{DuboisVioletteKernerMadore1990} M. Dubois-Violette, R. Kerner and J. Madore, Noncommutative differential geometry of matrix algebras, {\em J. Math. Phys.} {\bf 31}(2) (1990) 316--322.

\bibitem{FloystadVatne2009} G. Fløystad and J. E. Vatne, Artin-Schelter regular algebras of dimension five.  In {\em Algebra, geometry and mathematical physics}, volume 93 of {\em Banach Center Publ.}, pages 19--39. Polish Acad. Sci. Inst. Math., Warsaw, 2011.


\bibitem{GelfandKirillov1966} I. M. Gelfand and A. A. Kirillov, On fields connected with the enveloping algebras of Lie algebras, {\em Dokl. Akad. Nauk} {\bf 167}(3) (1966) 503--505.

\bibitem{GelfandKirillov1966b} I. M. Gelfand and A. A. Kirillov, Sur les corps liés aux algèbres enveloppantes des algèbres de Lie, {\em Publ. Math. IHES} {\bf 31} (1966) 5--19.

\bibitem{Karacuha2015} S. Karaçuha, Aspects of Noncommutative Differential Geometry, PhD Thesis, Universidade do Porto, Portugal (2015).

\bibitem{KaracuhaLomp2014} S. Karaçuha and C. Lomp, Integral calculus on quantum exterior algebras, {\em Int. J. Geom. Methods Mod. Phys.} {\bf 11}(04) (2014) 1450026.

\bibitem{KrauseLenagan2000} G. R. Krause and T. H. Lenagan, {\em Growth of Algebras and Gelfand-Kirillov Dimension}, Graduate Studies in Mathematics 22, AMS (2000).





\bibitem{LiWang2016} J. Li and X. Wang, Some five-dimensional Artin-Schelter regular algebras obtained by deforming a Lie algebra, {\em J. Algebra Appl.} {\bf 15}(4) (2016), 1650060.






\bibitem{ReyesSarmiento2022} A. Reyes and C. Sarmiento, On the differential smoothness of 3-dimensional skew polynomial algebras and diffusion algebras, {\em Internat. J. Algebra Comput.} {\bf 32}(3) (2022) 529--559.

\bibitem{Rogalski2023} D. Rogalski, {\em Artin-Schelter regular algebras}, in {\em Recent Advances in Noncommutative Algebra and Geometry}, Contemp. Math. \textbf{801} (2024), 195--242.



\bibitem{RubianoReyes2024DSBiquadraticAlgebras} A. Rubiano and A. Reyes, Smooth geometry of bi-quadratic algebras on three generators with PBW basis, (2024) \url{https://arxiv.org/abs/2408.16648}

\bibitem{RubianoReyes2024DSDoubleOreExtensions} A. Rubiano and A. Reyes, Smooth geometry of double extension regular algebras of type (14641), (2024) \url{https://arxiv.org/abs/2409.10264}

\bibitem{RubianoReyes2024DSSPBWKt} A. Rubiano and A. Reyes, Smooth geometry of skew PBW extensions over commutative polynomial rings I, {\em Bull. Iranian Math. Soc.} {\bf 51}(6) (2025), 77. 







\bibitem{Vancliff2024} M. Vancliff, An example of a quadratic AS-Regular algebra without any point modules, {\em Contemp. Math.} (2024), to appear.

\bibitem{WangWu2012} S. Q. Wang and Q. S. Wu, A class of AS-regular algebras of dimension five, {\em J. Algebra} {\bf 362} (2012) 117--144.

\bibitem{Woronowicz1987} S. L. Woronowicz, Twisted SU(2) Group. An Example of a Noncommutative Differential Calculus, {\em Publ. Res. Inst. Math. Sci.} {\bf 23}(1) (1987) 117--181.

\bibitem{ZhouLu2013} G.-S. Zhou and D.-M. Lu, Artin-Schelter regular algebras of dimension five with two generators, {\em J. Pure Appl. Algebra} {\bf 218}(5) (2014) 937--961.

\bibitem{ZhangZhang2008} Zhang, J. J., Zhang, J. (2008). Double Ore extensions. {\em J. Pure Appl. Algebra} 212(12):2668--2690.

\bibitem{ZhangZhang2009} Zhang, J.J., Zhang J. (2009). Double extension regular algebras of type (14641) {\em J. Algebra} 322(2):373--409.

\end{thebibliography}
\end{document}